\documentclass[reqno,11pt,a4paper]{amsart}

\usepackage[T1]{fontenc}
\usepackage{lmodern,microtype}
\usepackage[margin=28mm,headheight=14pt]{geometry}
\usepackage{amsmath,amssymb,mathtools,mathrsfs}
\usepackage{enumitem,xurl,aliascnt,needspace}
\usepackage[colorlinks=true,linkcolor=blue,citecolor=blue,urlcolor=blue]{hyperref}
\allowdisplaybreaks[2]
\numberwithin{equation}{section}

\newtheorem{theorem}{Theorem}[section]
\newaliascnt{proposition}{theorem}
\newtheorem{proposition}[proposition]{Proposition}
\aliascntresetthe{proposition}
\newaliascnt{lemma}{theorem}
\newtheorem{lemma}[lemma]{Lemma}
\aliascntresetthe{lemma}
\newaliascnt{corollary}{theorem}
\newtheorem{corollary}[corollary]{Corollary}
\aliascntresetthe{corollary}
\theoremstyle{definition}
\newaliascnt{definition}{theorem}

\aliascntresetthe{definition}
\theoremstyle{remark}
\newaliascnt{remark}{theorem}
\newtheorem{remark}[remark]{Remark}
\aliascntresetthe{remark}
\theoremstyle{plain}

\newcommand{\E}{\mathbb E}
\newcommand{\HS}{\mathrm{HS}}
\newcommand{\op}{\mathrm{op}}
\newcommand{\1}{\mathbf1}
\newcommand{\R}{\mathbb R}
\newcommand{\N}{\mathbb N}
\newcommand{\Z}{\mathbb Z}

\newcommand{\cB}{\mathcal B}
\newcommand{\dd}{\mathrm d}
\newcommand{\tr}{\operatorname{tr}}
\newcommand{\supp}{\operatorname{supp}}

\newcommand{\one}{\mathbf1}
\newcommand{\norm}[1]{\left\lVert#1\right\rVert}
\newcommand{\eqtag}[1]{\refstepcounter{equation}\label{#1}\eqno\textup{(\theequation)}}
\newcommand{\proofstep}[2]{\par\addvspace{0.65\baselineskip}\Needspace{4\baselineskip}\noindent\textbf{Step~#1. #2}\enspace\ignorespaces}

\title[Pointwise bilinear ergodic averages]{Pointwise convergence of bilinear ergodic averages for commuting transformations}

\author[Z. Guo]{Zhuowen Guo}
\author[R. Xiao]{Rongzhong Xiao}
\author[S. Zhang]{Shuhao Zhang}

\address[Zhuowen Guo, Rongzhong Xiao, and Shuhao Zhang]
{School of Mathematical Sciences, University of Science and
Technology of China, Hefei, Anhui 230026, P.R. China}

\email[Zhuowen Guo]{guozw0920@mail.ustc.edu.cn}
\email[Rongzhong Xiao]{xiaorz@mail.ustc.edu.cn}
\email[Shuhao Zhang]{yichen12@mail.ustc.edu.cn}

\date{}



\hypersetup{
  pdftitle={Pointwise convergence of bilinear ergodic averages for commuting transformations},
  pdfauthor={Zhuowen Guo, Rongzhong Xiao, Shuhao Zhang}
}

\begin{document}
\begin{abstract}
Let $(X,\cB,\mu)$ be a probability space and let $S,Q:X\to X$ be commuting probability-preserving transformations. Let $1\leq p,q\leq\infty$ satisfy $1/p+1/q\leq1$. We prove that, for every fixed $F\in L^p(\mu)$ and $G\in L^q(\mu)$, the averages
$$
\frac1N\sum_{n=1}^N F(S^nx)G(Q^nx)
$$
converge almost everywhere as $N\to\infty$.

Our proof introduces a new analytic approach that relies on neither the circle method nor characteristic-factor theory. Exploiting the triangular structure arising from commuting transformations, we combine matrix-energy estimates with Gaussian heat-flow arguments to establish pointwise variation inequalities for suitable smooth averages. We then use even and odd Gaussian kernels to recover the convergence of the original averages.
\end{abstract}

\maketitle

\section{Introduction}\label{sec:introduction}

\subsection{History and motivations}

Let $(X,\cB,\mu)$ be a probability space, and let $S,Q:X\to X$ be
commuting measure-preserving transformations. For measurable functions
$F,G$, define
$$
 B_N(F,G)(x)=\frac1N\sum_{n=1}^N F(S^nx)G(Q^nx).
 \eqtag{eq:bilinear-averages}
$$
We study the almost-everywhere convergence of the averages in
\eqref{eq:bilinear-averages} as $N\to\infty$ through all positive integers.

For a single observable, the theorems of von Neumann~\cite{vonNeumann}
and Birkhoff~\cite{Birkhoff} establish norm and pointwise convergence of
ordinary ergodic averages. Multiple averages arose from the connection
between recurrence and arithmetic configurations.
Furstenberg~\cite{Furstenberg} gave an ergodic-theoretic proof of
Szemer\'edi's theorem~\cite{Szemeredi} by establishing multiple recurrence
along arithmetic progressions. Furstenberg and
Katznelson~\cite{FurstenbergKatznelson} extended this recurrence principle
to commuting transformations and obtained a multidimensional
Szemer\'edi theorem. These developments led to the convergence problem
for products of functions sampled along several orbits. Recurrence alone
does not determine the behavior of those averages for arbitrary bounded
functions.

The norm convergence problem was resolved in increasing generality.
Conze and Lesigne~\cite{ConzeLesigne} proved $L^1$-convergence for $L^2$-inputs under two commuting invertible transformations,
and hence $L^2$-convergence for bounded inputs.
Bergelson and Leibman~\cite{BergelsonLeibman} extended the bilinear
$L^2$-convergence result for bounded inputs to pairs of invertible
measure-preserving transformations generating a nilpotent group,
and established the corresponding double recurrence theorem.
For arithmetic-progression averages in a single system, Host and Kra~\cite{HostKra} and,
independently, Ziegler~\cite{Ziegler} established norm convergence through
nilpotent characteristic factors. Tao~\cite{Tao} proved $L^2$-convergence for bounded inputs
and any finite family of commuting measure-preserving transformations.
Austin~\cite{AustinNorm} gave an ergodic-theoretic proof using
pleasant extensions and treated commuting $\Z^r$-actions along arbitrary
F\o lner sequences and their translates. Walsh~\cite{Walsh} proved norm
convergence for polynomial averages generated by nilpotent groups of
transformations. Austin's construction of sated
extensions~\cite{AustinPleasant} further developed the
extension method for identifying characteristic factors.

Polynomial sampling introduces another difficulty even for a single
observable. Bourgain~\cite{BourgainPolynomial} proved
almost-everywhere convergence of
$$
 \frac1N\sum_{n=1}^N f(T^{P(n)}x)
$$
for $f\in L^p$, $1<p<\infty$, and a nonconstant $P\in\Z[n]$, where
$T$ is an invertible measure-preserving transformation of a probability
space. His proof combines maximal and oscillation estimates with the
circle method. Mirek, Stein and Trojan~\cite{MST}
proved $L^p$-bounds for pointwise $\rho$-variation of polynomial averages
of discrete Radon type, for $1<p<\infty$ and $\rho>2$, with constants
independent of the polynomial coefficients. Their ergodic estimates
include a single function evaluated under compositions of polynomial
iterates of commuting invertible measure-preserving transformations.
Ionescu, Magyar, Mirek and Szarek~\cite{IMMS} established norm and
almost-everywhere convergence of symmetric single-function polynomial
averages for $L^p$-inputs with $1<p<\infty$, under invertible
measure-preserving transformations generating a nilpotent group of step two.

Independent averaging parameters lead to a different extension.
Bourgain, Mirek, Stein and Wright~\cite{BMSW} proved
norm and almost-everywhere convergence of the rectangular averages
$$
 \frac1{N_1\cdots N_k}
 \sum_{n_1=1}^{N_1}\cdots\sum_{n_k=1}^{N_k}
       f\bigl(T^{P(n_1,\ldots,n_k)}x\bigr)
$$
as $\min_jN_j\to\infty$, for $k\ge2$, $f\in L^p$ with
$1<p<\infty$, and $P\in\Z[n_1,\ldots,n_k]$ satisfying $P(0)=0$.
Here $T$ is invertible and measure-preserving.
Their proof uses multi-parameter maximal and oscillation inequalities.
These results concern one function sampled along polynomial orbits.
In \eqref{eq:bilinear-averages}, two functions are sampled along different
orbits at the same averaging index.

For products of functions, Bourgain's double recurrence
theorem~\cite{Bourgain} gives almost-everywhere convergence of the averages
$$
 \frac1N\sum_{n=1}^N F(T^{an}x)G(T^{bn}x)
$$
for bounded $F,G$ and integer speeds $a,b$, when $T$ is invertible and
measure-preserving.
Strictly ergodic models provide another approach.
For each fixed $d$, Huang, Shao and Ye~\cite{HSYModels}
constructed a topological model of an invertible ergodic standard system
whose arithmetic-progression orbit closure is minimal and uniquely
ergodic under the group generated by
$$
 T\times\cdots\times T
 \quad\text{and}\quad T\times T^2\times\cdots\times T^d.
$$
They also proved almost-everywhere convergence of
$$
 \frac1N\sum_{n=0}^{N-1}\prod_{j=1}^d f_j(T^{jn}x)
$$
for every $d$ and bounded inputs when $T$ is ergodic and measurably distal.
Donoso and Sun obtained the pointwise result for ergodic measurably
distal actions generated by two commuting
transformations~\cite{DonosoSun}, and then for a finite
family of commuting transformations~\cite{DonosoSunMultiple}.
Here ergodicity refers to the generated group action. The product-system
formulation explains a difficulty in general: the average in
\eqref{eq:bilinear-averages} is an ergodic average of $F\otimes G$ under
$S\times Q$, evaluated at $(x,x)$. A full-measure set for $\mu\times\mu$
need not contain almost every point of the diagonal with respect to its
copy of $\mu$.

For multiple polynomial averages, Krause, Mirek and Tao~\cite{KMT}
proved pointwise convergence of the averages
$$
 \frac1N\sum_{n=1}^N F(T^nx)G(T^{P(n)}x)
$$
when $P\in\Z[n]$ has degree at least two and the inputs are bounded.
Kosz, Mirek, Peluse, Wan and Wright~\cite{KMPWW} treated bounded
inputs and commuting invertible transformations with polynomial iterates of
pairwise distinct degrees. The two linear iterates in
\eqref{eq:bilinear-averages} are outside that distinct-degree hypothesis.

\subsection{Main result}

We prove almost-everywhere convergence for every commuting pair. The
integrability range includes the endpoint pairs $(1,\infty)$ and
$(\infty,1)$.

\begin{theorem}\label{thm:main}
Let $(X,\cB,\mu)$ be a probability space and let $S,Q:X\to X$ be
commuting probability-preserving transformations. Let
$1\leq p,q\leq\infty$ satisfy $1/p+1/q\leq1$. For every fixed
$F\in L^p(\mu)$ and $G\in L^q(\mu)$, the averages $$\frac1N\sum_{n=1}^N F(S^nx)G(Q^nx)$$
converge almost everywhere as $N\to\infty$.
\end{theorem}

The limit in this theorem is generally not constant. For instance,
when $S=Q$ the averages are the ordinary ergodic averages of $FG$.
We give a product-of-integrals conclusion under joint ergodicity
in Corollary~\ref{cor:joint}. This includes an application to products of
weakly mixing systems in which each coordinate is evaluated at
at most two distinct nonzero speeds; see Corollary~\ref{cor:independent}.
The integrability range in Theorem~\ref{thm:main} is sharp for the class
of all commuting pairs, as shown in Proposition~\ref{prop:range}.

\subsection{Proof outline of the main result}

We prove the bounded-input case by controlling selected increments
of the forward averages. The analytic estimate is for finite arrays;
an approximation of the averaging kernel transfers it directly
along nonnegative orbits.

\proofstep{1}{A finite-array estimate.}
For finitely supported real arrays $|A_v|\le1$, put
$N_v=\#\supp A_v$. At each entry of $A_0$, select at most $m$
disjoint scale intervals and coefficients of modulus at most one.
Proposition~\ref{prop:lattice} bounds the resulting triangular form by
$$
 |\Lambda^\psi(A_0,A_1,A_2)|
       \le C\sqrt m\,(N_0+N_1+N_2),
 \qquad \psi=g_3^{(r)},\quad r=3,4,5.
$$
The selected kernels are
$I^\psi_{a,b}=\int_{a^2}^{b^2}\psi_{\sqrt s}\,ds/s$.
A finite sequence of rank-one insertions bounds the cyclic matrix
trace by row costs. Pinching the two column blocks then controls
the column costs by the row costs, so the total row energy has a
negative heat derivative. Multiplying the changing array by
$\delta=m^{-1/2}$ bounds all energy jumps by
$Cm\delta^2N_0=CN_0$. Integration and division by $\delta$
give the estimate.

\proofstep{2}{Two kernels with controlled increments.}
Put $\gamma=g_3$, $K=\gamma-\tfrac32\gamma''$ and $H=\gamma'''$.
Their scale derivatives give
$$
 K_b-K_a=-\frac94I^{\gamma^{(4)}}_{a,b},\qquad
 H_b-H_a=\frac32\bigl(I^{\gamma^{(3)}}_{a,b}
                         +I^{\gamma^{(5)}}_{a,b}\bigr).
$$
Thus every finite linear combination $P$ of dilates of $K,H$
has the same selected-increment estimate, with a constant depending
on $P$ and a fixed lower scale threshold.

\proofstep{3}{Approximate the forward kernel and transfer.}
On the positive half-line, a bounded annihilator of the dilates
of either kernel gives a Gaussian transform satisfying
$aF'(a)+\tfrac32F(a)=0$. Its growth bound forces $F=0$;
polynomial density and Hahn--Banach give density of both dilation
families. The evenness of $K$ and oddness of $H$ then approximate
$b=\1_{(0,1]}$ in $L^1(\R)$.
An explicit sampling estimate gives, for any $\varepsilon>0$,
a profile $P$ as above and an integer $M$ such that
$$
 \sum_{n\in\Z}|b_N(n)-P_N(n)|\le\varepsilon
                         \qquad(N\ge M).
$$
Transfer the finite-array estimate on positive orbit boxes,
absorbing the kernel error before identifying the orbit sums.
Lemma~\ref{lem:transfer} gives
$$
 \int D_m^M(B(F,G))\,d\mu\le C_P\sqrt m+4m\varepsilon
$$
for normalized bounded real inputs. Here $D_m^M$ is the largest
sum of absolute increments over at most $m$ ordered disjoint
integer intervals with endpoints at least $M$.

\proofstep{4}{The Cauchy criterion and general inputs.}
The finite initial range contributes at most $2M$ to $D_m=D_m^1$.
First let $m\to\infty$ with $P,M,\varepsilon$ fixed, and then
let $\varepsilon\downarrow0$. This gives
$\int D_m(B(F,G))\,d\mu=o(m)$.
Infinitely many increments of any fixed positive size would
contradict this bound, so the forward averages converge almost
everywhere. Real and imaginary parts give bounded complex inputs.
Truncation, H\"older's inequality and Birkhoff's theorem give the
full stated range of $p,q$, including the infinite endpoints.

\subsection*{Structure of the paper}
Section~\ref{sec:preliminaries} fixes the notation and recalls the elementary
background. Section~\ref{sec:matrix-heat} fixes the matrix and Gaussian notation.
Section~\ref{sec:counts} establishes the finite-lattice interval estimate, and
Section~\ref{sec:transference} transfers it to the dynamical setting and proves
the variation bounds. Section~\ref{sec:recovery} recovers the forward averages.
Section~\ref{sec:extensions} proves Theorem~\ref{thm:main}, the optimality of its
integrability range, and the applications to joint ergodicity.

\section{Preliminaries}\label{sec:preliminaries}

We write $\N=\{1,2,\ldots\}$ and $\N_0=\N\cup\{0\}$.
Lebesgue measure is used on Euclidean spaces and counting measure
on $\Z$ and $\Z^2$. For a finite set $E$, $\#E$ is its cardinality;
$\one_E$ is its indicator. The support of an array is the set of
its nonzero entries. Integrals of nonnegative functions are understood
in the extended sense until their finiteness is established.
The letter $C$ denotes a positive absolute constant that may change
from line to line. Unless explicitly stated, it is independent of
the finite index sets, arrays, selected endpoints and transformations.

Function spaces are complex. We prove the analytic estimate and
the bounded dynamical case for real inputs, and then split complex
inputs into real and imaginary parts. All matrices in the analytic
argument act between finite-dimensional real Euclidean spaces, and
$R^*$ denotes transpose. We allow different sizes at the three vertices.
The operator norm is the Euclidean operator norm, and
$$
 \langle A,B\rangle_{\HS}=\tr(A^*B),\qquad
 \|A\|_{\HS}^2=\sum_{i,j}|A_{ij}|^2.
$$
Horizontal concatenation is written $[A\ B]$.
The row and column norms are Euclidean norms. For symmetric matrices,
$P\ge0$ means positive semidefinite and $P\le Q$ means $Q-P\ge0$.
Positive fractional powers are defined by orthogonal spectral calculus;
in particular $P^{3/2}=P\sqrt P$. The matrix support
$\supp P=(\ker P)^\perp$ is a subspace, whereas the support of an
array is a set of entries. A matrix $H$ supported on $P$ satisfies
$H=\pi_P H\pi_P$, where $\pi_P$ is the orthogonal projection onto
$\supp P$. These conventions include the zero-dimensional support.

\subsection{Measure-preserving systems}

A measurable map $T:X\to X$ on a probability space $(X,\cB,\mu)$ is
measure-preserving if $\mu(T^{-1}E)=\mu(E)$ for every $E\in\cB$.
When $T$ is invertible, its inverse is also assumed measurable.
Negative powers of a transformation are used only in this case.
Statements about $L^p$ functions are understood up to null sets, and
every pointwise conclusion concerns a fixed choice of the input functions.
We refer to \cite[Chapter~1]{Walters} for the ergodic-theoretic background.

For a sub-sigma-algebra $\mathcal A\subseteq\cB$ and $f\in L^1(\mu)$,
the conditional expectation $\mathbb E_\mu(f\mid\mathcal A)$ is the
$\mathcal A$-measurable integrable function characterized, up to a null
set, by
$$
 \int_E\mathbb E_\mu(f\mid\mathcal A)\,\dd\mu
       =\int_E f\,\dd\mu\qquad(E\in\mathcal A).
$$
Conditional expectation is positive and commutes with monotone limits
whenever the limiting function is integrable.
In particular, if $0\le h_m\le h_1\in L^1(\mu)$ and $h_m\downarrow0$
almost everywhere, then
$$
 \mathbb E_\mu(h_m\mid\mathcal A)\downarrow0
       \quad\text{almost everywhere}.
 \eqtag{eq:conditional-tails}
$$
See \cite[Chapter~9]{Williams}; the decreasing form follows by applying
conditional monotone convergence to $h_1-h_m$.
The invariant sigma-algebra of $T$ is
$$
 \mathcal I_T=\{E\in\cB:\mu(T^{-1}E\mathbin{\triangle}E)=0\},
$$
where $\triangle$ denotes symmetric difference.

\begin{theorem}[Birkhoff and von Neumann]\label{thm:ergodic-prelim}
Let $T$ be a measure-preserving transformation of a probability space.
For every $f\in L^1(\mu)$,
$$
 \frac1N\sum_{n=1}^N f(T^nx)
       \longrightarrow\mathbb E_\mu(f\mid\mathcal I_T)(x)
       \quad\text{for almost every }x.
$$
For $f\in L^2(\mu)$, the convergence also holds in $L^2(\mu)$.
\end{theorem}

The two convergence statements are the pointwise and mean ergodic
theorems; see \cite{Birkhoff,vonNeumann} and
\cite[Section~1.6]{Walters}.
A transformation $T$ is ergodic if every set in $\mathcal I_T$ has
measure zero or one. In that case,
$\mathbb E_\mu(f\mid\mathcal I_T)=\int f\,\dd\mu$ almost everywhere.
It is weakly mixing if $T\times T$ is ergodic on
$(X\times X,\mu\times\mu)$. Nonzero powers and finite products of
invertible weakly mixing transformations are weakly mixing; see
\cite[Chapter~1]{Walters}.

For transformations $S,Q$ and functions $F,G$ on $X$, we write
$$
 (S\times Q)(x,y)=(Sx,Qy),\qquad
 (F\otimes G)(x,y)=F(x)G(y).
$$
We use analogous notation for finite products, whereas $SQ$ denotes
composition on $X$. The commuting transformations $S,Q$ are jointly
ergodic if, for all bounded $F,G$, the averages in \eqref{eq:bilinear-averages}
satisfy
$$
 B_N(F,G)\longrightarrow
       \left(\int F\,\dd\mu\right)\left(\int G\,\dd\mu\right)
       \quad\text{in }L^2(\mu).
$$
Corollary~\ref{cor:joint} supplies sufficient conditions for this property.

\section{Finite matrix estimates and Gaussian notation}\label{sec:matrix-heat}

\subsection{Gaussian kernels}
For a function $f$ on $\R$ and $R>0$, we set its dilation by $f_R(t):=R^{-1}f(t/R)$.
 We use
$$
 g_s(t)=(2\pi s)^{-1/2}e^{-t^2/(2s)},\qquad s>0,
 \qquad (f*g)(t)=\int_\R f(t-u)g(u)\,\dd u.
$$
The subscript in $g_s$ also is its variance. Primes denote derivatives in the real
variable. In particular,
$$
 g_s*g_t=g_{s+t},\qquad 2\partial_sg_s=g_s'',\qquad
 (g_1^{(r)})_{\sqrt s}=s^{r/2}g_s^{(r)}.
$$
A Schwartz function is a smooth function whose derivatives decay
faster than every inverse power: $\sup_t(1+|t|)^a|f^{(b)}(t)|<\infty$
for all $a,b\in\N_0$. All Gaussian derivatives have this property.
Their integrability, convolution identities and sampling bounds will
be used below.

\subsection{Matrix energies}

For $P\ge0$, put $\Phi(P)=\tr(P^{3/2})$.
Write $\mathsf L_A(Y)=AY$ and $\mathsf R_A(Y)=YA$.
If a symmetric matrix $H$ vanishes on $\ker P$ on both sides, define
$$
 b_P(H)=\frac32\langle H,
       (\mathsf L_{\sqrt P}+\mathsf R_{\sqrt P})^{-1}H\rangle_{\HS}.
$$
The inverse is taken on matrices acting on $\supp P=(\ker P)^\perp$;
$b_0(0)=0$. Let $B_P$ be the associated symmetric bilinear form.
In an orthonormal eigenbasis of $P$, restricted to its positive eigenvalues,
$$
 b_P(H)=\frac32\sum_{i,j}
                  \frac{|H_{ij}|^2}{\sqrt{p_i}+\sqrt{p_j}}.
$$

\begin{lemma}\label{lem:calculus}
The form $B_P$ is the Hessian of $\Phi$ on each fixed support.
If $P\le Q$ and $H$ is supported on $P$, then $b_Q(H)\le b_P(H)$.
For measurable integrable matrix pairs $(P_\omega,H_\omega)$ on a
probability space $(\Omega,\nu)$, where $P_\omega\ge0$ and each
$H_\omega$ is symmetric and supported on $P_\omega$,
$$
 b_{\int P_\omega\,d\nu}
       \left(\int H_\omega\,d\nu\right)
 \le \int b_{P_\omega}(H_\omega)\,d\nu.
$$
Furthermore, for symmetric $L$,
$$
 b_{R^*R}(R^*LR)\le\frac34\|L\|_{\op}^2\Phi(R^*R).
$$
The map $R\mapsto\Phi(R^*R)$ is continuously differentiable on
the whole matrix space. For a line segment $R(t)=R_-+t\Delta R$,
where $\Delta R=R_+-R_-$,
$$
 |\Phi(R_+^*R_+)-\Phi(R_-^*R_-)|
 \le 3\int_0^1\sum_{i,j}|\Delta R_{ij}|
             \|\operatorname{row}_iR(t)\|_2
             \|\operatorname{col}_jR(t)\|_2\,dt.
$$
\end{lemma}

\begin{proof}
On a fixed support, square root is smooth by the implicit function
theorem: at $A>0$, the derivative $X\mapsto AX+XA$ of $A\mapsto A^2$
is invertible. Write $A=\sqrt P$. Differentiating $A^2=P$ gives
$AX+XA=H$ for $X=d(\sqrt P)[H]$. Hence
$d\Phi_P(H)=\frac32\tr(\sqrt P H)$, and another derivative gives
the displayed Hessian.

Square root preserves order. Indeed, inversion reverses order and
$$
 \sqrt P=\frac1\pi\int_0^\infty
             P(P+tI)^{-1}\,\frac{dt}{\sqrt t}.
$$
The order assertion for $b$ now follows from
$$
 b_P(H)=\frac32\sup_{Y=Y^*}
       \{2\tr(HY)-2\tr(\sqrt P\,Y^2)\}.
$$
This formula is obtained by completing the square. On a larger space
the same supremum gives the supported cost; directions inside the
kernel do not contribute. Thus it also proves the stated comparison
when the supports differ.

For averaging, first put $\bar P=\int P_\omega\,d\nu$ and
$\bar H=\int H_\omega\,d\nu$. If $z\in\ker\bar P$, then
$\int\langle P_\omega z,z\rangle\,d\nu=0$. Positivity gives
$P_\omega z=0$ and therefore $H_\omega z=0$ almost everywhere.
Applying this to a basis of $\ker\bar P$ shows that $\bar H$
is supported on $\bar P$; no common support is required.
Set $A=\int\sqrt{P_\omega}\,d\nu$.
Vector-valued Cauchy--Schwarz gives $A^2\le\bar P$,
so $\sqrt{\bar P}\ge A$.
Insert this inequality into the full-space supremum formula and
take the supremum after the integral. This proves the assertion.
The right side may be infinite.

For the inserted bound, use the singular values $\rho_i$ of $R$.
The cost is
$$
 \frac32\sum_{i,j}
      \frac{\rho_i^2\rho_j^2}{\rho_i+\rho_j}|L_{ij}|^2.
$$
Here $L$ is compressed to the left singular vectors of $R$.
The inequality
$x^2y^2/(x+y)\le(x^3+y^3)/4$ and the row and column bounds for $L$
give the claim.

Finally, put $\Psi(R)=\Phi(R^*R)$ and, for $\varepsilon>0$, put
$\Psi_\varepsilon(R)=\tr((R^*R+\varepsilon I)^{3/2})$.
Positive definite calculus gives
$$
 \nabla\Psi_\varepsilon(R)=3R\sqrt{R^*R+\varepsilon I}.
$$
If $R$ has $n$ columns, scalar spectral calculus gives
$$
 \begin{aligned}
 0\le\Psi_\varepsilon(R)-\Psi(R)
   &\le\frac32n\varepsilon
                    (\|R\|_{\op}^2+\varepsilon)^{1/2},\\
 \|\nabla\Psi_\varepsilon(R)-3R|R|\|_{\HS}
   &\le3\sqrt\varepsilon\,\|R\|_{\HS},
 \qquad |R|=\sqrt{R^*R}.
 \end{aligned}
$$
Thus the functions and their gradients converge uniformly on bounded
sets. Integrating the gradients along any segment and passing to the
limit yields
$$
 \Psi(R+L)-\Psi(R)
       =3\int_0^1\langle (R+tL)|R+tL|,L\rangle_{\HS}\,dt.
$$
The limiting gradient is continuous, so this identity proves that
$\Psi$ is $C^1$ on the whole matrix space, with
$\nabla\Psi(R)=3R|R|$, including at changes of rank.
For each entry, Cauchy--Schwarz gives
$$
 |(R|R|)_{ij}|
      \le\|\operatorname{row}_iR\|_2
           \|\operatorname{col}_jR\|_2,
$$
because the $j$-th columns of $|R|$ and $R$ have the same norm.
Apply this bound to the segment identity with $R=R_-$ and
$L=\Delta R$ to obtain the asserted energy-change estimate.
\end{proof}

\begin{lemma}[Pinching the two column blocks]\label{lem:pinching}
If $R=[\,A\quad B\,]$ and $N$ is symmetric on its row space, then
$$
 b_{A^*A}(A^*NA)+b_{B^*B}(B^*NB)
          \le b_{R^*R}(R^*NR).
$$
\end{lemma}
\begin{proof}
Put $T=R^*R$, $H=R^*NR$ and
$U=\operatorname{diag}(I,-I)$ on the two column spaces. Then
$$
 \frac{T+UTU}{2}
   =\operatorname{diag}(A^*A,B^*B),\qquad
 \frac{H+UHU}{2}
   =\operatorname{diag}(A^*NA,B^*NB).
$$
Lemma~\ref{lem:calculus} applies to these two supported pairs even
if their supports differ. Its averaging inequality, orthogonal
conjugation invariance of $b$, and additivity over direct sums give
the claim.
\end{proof}

Let $H_0,H_1,H_2$ be finite-dimensional real Euclidean spaces.
For compatible matrices $W_v:H_{v+1}\to H_v$, with indices modulo $3$, put
$$
 R_v=[\,W_v\quad W_{v-1}^*\,],\qquad
 T_v=R_v^*R_v,\qquad S_v=R_vR_v^*.
$$
The two Grams have the same nonzero eigenvalues, with multiplicity.

\begin{lemma}\label{lem:triangle}
For symmetric $D_v$, with $D_*=\max_v\|D_v\|_{\op}$,
$$
 |\tr(D_0W_0D_1W_1D_2W_2)|
       \le D_*\sum_{v=0}^2 b_{T_v}(R_v^*D_vR_v).
$$
\end{lemma}
\begin{proof}
The assertion on a zero-dimensional total space is immediate.
Form the real symmetric block matrices
$$
 M=\begin{pmatrix}
 0&W_0&W_2^*\\ W_0^*&0&W_1\\ W_2&W_1^*&0
 \end{pmatrix},\qquad
 D=\operatorname{diag}(D_0,D_1,D_2).
$$
Let $P_v$ be the orthogonal projection onto the $v$-th block.
The matrices $MP_vM$ and $L_v=MP_vDM$ are $T_v$ and its insertion,
extended by zero. Since $P_vD=DP_v$, each $L_v$ is symmetric.
Also $MP_vM\le M^2$ and $\sum_vL_v=MDM$.
Order comparison and Cauchy--Schwarz for the nonnegative quadratic
form $b_{M^2}$ therefore give
$$
 \mathcal E:=\sum_v b_{T_v}(R_v^*D_vR_v)
 \ge\sum_v b_{M^2}(L_v)
 \ge\frac13 b_{M^2}(MDM).
$$

Work now in an orthonormal eigenbasis $e_1,\ldots,e_d$ of $M$,
ordered so that $q_i=|m_i|$ is nonincreasing. The matrix of $D$
in this basis is still symmetric and has norm $D_*$.
Insert the eigenvalues one at a time. For $0\le t\le1$, define
the symmetric matrices and the real-valued functions
$$
 X_i(t)=t m_i e_ie_i^*+\sum_{j>i}m_je_je_j^*,\qquad
 F_i(t)=\tr((DX_i(t))^3).
$$
They satisfy $X_1(1)=M$, $X_i(0)=X_{i+1}(1)$ for $i<d$,
and $X_d(0)=0$. Cyclicity of the trace gives, with
$v_i(t)=X_i(t)De_i$,
$$
 F_i'(t)=3m_i e_i^*DX_i(t)DX_i(t)De_i
        =3m_i\langle Dv_i(t),v_i(t)\rangle.
$$
The coordinates of $v_i(t)$ are zero for $j<i$, equal to
$t m_iD_{ii}$ at $i$, and equal to $m_jD_{ji}$ for $j>i$.
Thus
$$
 \|v_i(t)\|_2^2
       =t^2q_i^2D_{ii}^2+\sum_{j>i}q_j^2D_{ij}^2.
$$
Use $|\langle Dv,v\rangle|\le D_*\|v\|_2^2$ and
$\int_0^1t^2\,dt=1/3$, then telescope the successive insertions. This gives
$$
 \begin{aligned}
 |\tr((DM)^3)|
 &\le\sum_i\int_0^1|F_i'(t)|\,dt\\
 &\le D_*\left[
      \sum_iq_i^3D_{ii}^2
           +3\sum_{i<j}q_iq_j^2D_{ij}^2\right].
 \end{aligned}
$$
The eigenbasis formula for $b$ is
$$
 b_{M^2}(MDM)
   =\frac34\sum_iq_i^3D_{ii}^2
       +3\sum_{i<j}\frac{q_i^2q_j^2}{q_i+q_j}D_{ij}^2.
$$
Every ratio with $q_i+q_j=0$ is defined as zero, since $MDM$
vanishes on $\ker M$. The diagonal coefficients in the preceding
bracket are at most those in $2b_{M^2}(MDM)$. For $i<j$,
$q_i+q_j\le2q_i$, so its off-diagonal coefficients satisfy
$$
 3q_iq_j^2\le\frac{6q_i^2q_j^2}{q_i+q_j}.
$$
Consequently
$$
 |\tr((DM)^3)|\le2D_*b_{M^2}(MDM)\le6D_*\mathcal E.
$$
Zero eigenvalues contribute zero throughout, so no invertibility
assumption is needed.

The zero diagonal blocks of $M$ leave only the two oriented
three-cycles, each with three starting blocks. Cyclicity and transpose
symmetry therefore give
$$
 \begin{aligned}
 \tr((DM)^3)
 &=3\tr(D_0W_0D_1W_1D_2W_2)
       +3\tr(D_0W_2^*D_2W_1^*D_1W_0^*)\\
 &=6\tr(D_0W_0D_1W_1D_2W_2).
 \end{aligned}
$$
Division by $6$ proves the claim.
\end{proof}

\section{The bounded lattice estimate}\label{sec:counts}

For $\psi=g_3^{(r)}$, $r=3,4,5$, and $0<a<b$, define
$$
 I^\psi_{a,b}(t)=\int_{a^2}^{b^2}\psi_{\sqrt s}(t)\,\frac{ds}{s}.
$$

\begin{proposition}\label{prop:lattice}
Let $A_0,A_1,A_2$ be finitely supported real arrays on $\Z^2$,
with $|A_v|\le1$, and put $N_v=\#\supp A_v$.
Let $m\in\N$. At each $(i,j)$ choose at most $m$ disjoint intervals
$(a_\ell,b_\ell)$ with $4\le a_\ell<b_\ell$, and real coefficients
$|\beta_\ell|\le1$. Then
$$
 \left|\sum_{i,j,k}A_0(i,j)A_1(j,k)A_2(k,i)
       \sum_\ell\beta_\ell(i,j)
            I^\psi_{a_\ell(i,j),b_\ell(i,j)}(i+j+k)\right|
 \le C\sqrt m\,(N_0+N_1+N_2).
$$
The constant is independent of all endpoints and supports.
\end{proposition}

\begin{proof}
Set $\delta=m^{-1/2}$ and
$$
 A_0^s(i,j)=\delta A_0(i,j)\sum_\ell\beta_\ell(i,j)
          \one_{(a_\ell(i,j)^2,b_\ell(i,j)^2)}(s),
 \qquad A_1^s=A_1,\quad A_2^s=A_2.
$$
Outside $\supp A_0$ use empty lists. Only finitely many endpoints
then matter. Disjointness gives $|A_0^s|\le\delta$ and
$$
 \sum_\tau|\Delta A_0^\tau(i,j)|
       \le2m\delta\,\one_{\supp A_0}(i,j),
 \qquad \Delta A_0^\tau=A_0^{\tau+}-A_0^{\tau-}.
$$
Choose $s_-=2$ and $s_+$ beyond all relevant endpoint squares.
All matrices can be restricted to one finite index set containing
all coordinates of the supports.

\smallskip
\noindent\emph{Gaussian matrices.}
For centers $p_v$ and variances $\sigma_v$, put
$$
 w_v(i)=g_{\sigma_v}(i-p_v),\qquad
 W_v(i,j)=A_v^s(i,j)\sqrt{w_v(i)w_{v+1}(j)},
 \qquad e_v=\Phi(T_v)=\Phi(S_v).
$$
Write $C_{v,v+1}=W_v$ and $C_{v,v-1}=W_{v-1}^*$ for the two
blocks of $R_v$. Define
$$
 \mathsf N_v=\operatorname{diag}_i\frac{i-p_v}{\sigma_v},\qquad
 U_v=R_v^*\mathsf N_vR_v,\qquad V_v=b_{T_v}(U_v),\qquad
 P_{v,w}=C_{v,w}\mathsf N_wC_{v,w}^*.
$$
Also put
$$
 Y_{v,w}=b_{S_v}(P_{v,w}),\qquad
 Z_v=B_{S_v}(P_{v,v+1},P_{v,v-1}).
$$
These insertions are supported on $T_v$ and $S_v$, respectively:
$U_v$ vanishes on $\ker R_v$, and $\ker S_v$ is contained in
$\ker C_{v,w}^*$.

With the arrays fixed, differentiation of a Gaussian density gives
$\partial_pw=Nw$, $\partial_p^2w=(N^2-\sigma^{-1})w$ and
$2\partial_\sigma w=\partial_p^2w$, where $N=(i-p)/\sigma$.
For a row variable, $T_v$ is linear in $w_v$; for a column
variable $w\ne v$, $S_v$ is linear in $w_w$. Their respective
supports are fixed, since all weights are positive.
The Hessian chain rule gives
$$
 \begin{aligned}
 2\partial_{\sigma_v}e_v&=\partial_{p_v}^2e_v-V_v,\\
 2\partial_{\sigma_w}e_v&=\partial_{p_w}^2e_v-Y_{v,w},\\
 \partial_{p_{v+1}}\partial_{p_{v-1}}e_v&=Z_v.
 \end{aligned}
$$
The last identity holds because
$S_v=C_{v,v+1}C_{v,v+1}^*+C_{v,v-1}C_{v,v-1}^*$ and its two
summands depend on different column centers, so its mixed derivative
is zero.

\smallskip
\noindent\emph{One total row energy.}
The column costs are controlled by the row costs at the same
weights. Indeed, $C_{v,w}C_{v,w}^*\le S_v$, so comparison and
Lemma~\ref{lem:pinching}, applied to the two blocks of $R_w$, give
$$
 \begin{aligned}
 \sum_{v\ne w}Y_{v,w}
 &\le\sum_{v\ne w}
       b_{C_{v,w}C_{v,w}^*}
          (C_{v,w}\mathsf N_wC_{v,w}^*)\\
 &\le V_w.
 \end{aligned}
 \eqtag{eq:column-row-cost}
$$
For the second line, use $C_{w,v}=C_{v,w}^*$: the two terms
in the first line are precisely the two diagonal-block costs
obtained by pinching $R_w^*\mathsf N_wR_w$.

Fix a distinguished vertex $v_*$. Set
$$
 \alpha_{v_*}=\frac78,\qquad
 \alpha_w=\frac98\quad(w\ne v_*),\qquad
 \sigma_w=\alpha_ws,
$$
and define
$$
 \langle H\rangle_s=
     \int_{\R^3}g_{s/4}(p_0+p_1+p_2)H(p,s)\,dp,\qquad
 J^{(v_*)}(s)=\left\langle\sum_{v=0}^2e_v\right\rangle_s,\qquad
 \widetilde V_w^{(v_*)}(s)=\langle V_w\rangle_s.
$$
The bracket is not a probability measure on $\R^3$.
Each choice of $v_*$ uses the indicated weights, and
$J^{(v_*)}\ge0$.

We justify the differentiations and integrations by parts.
Lemma~\ref{lem:calculus} makes $R\mapsto\Phi(R^*R)$ continuously
differentiable on the whole matrix space. Hence first derivatives
in all Gaussian parameters, including the scale chain rule, are
valid between change scales. For a second derivative in one center,
we use the fixed-support Gram identified above. When both column
centers of $R_v$ vary with its row weights fixed, each block
$C_{v,w}$ is right-multiplied by an invertible positive diagonal
matrix. Thus its adjoint kernel is unchanged, and
$$
 \ker S_v=\ker C_{v,v+1}^*\cap\ker C_{v,v-1}^*
$$
is fixed as well. The fixed-support Hessian therefore also justifies
the mixed column derivative. These are all the second derivatives
used below.
The insertion bound in
Lemma~\ref{lem:calculus}, its first derivative formula, and
Cauchy--Schwarz for $B_P$ bound all derivatives used below by
$$
 C_A(1+|p_0|+|p_1|+|p_2|)^2e_v
$$
on compact positive scale intervals. Here $C_A$ may depend on the
fixed arrays and the compact scale interval. The first derivative inserts
a score, and the second inserts
$\mathsf N_w^2-\sigma_w^{-1}I$ and a Hessian term.
Also, on such intervals,
$$
 e_v\le C_A\sum_{i,j}|A_v^s(i,j)|^2w_v(i)w_{v+1}(j)
      +C_A\sum_{k,i}|A_{v-1}^s(k,i)|^2w_{v-1}(k)w_v(i).
$$
This follows from $\Phi(R_v^*R_v)\le\|R_v\|_{\op}\|R_v\|_{\HS}^2$
and the Gaussian grid sums. After multiplication by the bracket
weight, each summand is a product of three Gaussians in independent
linear coordinates: its two endpoint centers and $p_0+p_1+p_2$.
Polynomial multiples have an integrable majorant locally uniform in
scale, obtained by increasing the variances to fixed larger values.
The same holds for derivatives of the bracket weight. Thus all
following differentiations under the integral and integrations by
parts are justified. Changes of rank at a change scale are treated
by the line-segment formula.

The pure second center derivatives can all be transferred to the
bracket weight. Their coefficients sum to
$\sum_w\alpha_w=25/8$.
Since $2\partial_sg_{s/4}=\frac14g_{s/4}''$,
differentiating the weight adds $1/4$ to this coefficient.
Write
$$
                    c=\sum_w\alpha_w+\frac14=\frac{27}{8}.
$$
For each $e_v$, transferring one derivative back to each of its
two column centers gives
$$
 \int_{\R^3}g_{s/4}''(p_0+p_1+p_2)e_v\,dp
     =\langle\partial_{p_{v+1}}\partial_{p_{v-1}}e_v\rangle_s
     =\langle Z_v\rangle_s.
$$
The heat identities therefore give
$$
 \begin{aligned}
 2(J^{(v_*)})'
 &=-\sum_w\alpha_w\widetilde V_w^{(v_*)}
      +c\left\langle\sum_vZ_v\right\rangle_s
      -\sum_w\alpha_w
                  \left\langle\sum_{v\ne w}Y_{v,w}\right\rangle_s\\
 &\le-\sum_w\alpha_w\widetilde V_w^{(v_*)}
      +\sum_w\left(\frac c2-\alpha_w\right)
                  \left\langle\sum_{v\ne w}Y_{v,w}\right\rangle_s\\
 &\le\sum_w\left(\frac c2-2\alpha_w\right)
                              \widetilde V_w^{(v_*)}\\
 &=-\frac1{16}\widetilde V_{v_*}^{(v_*)}
       -\frac9{16}\sum_{w\ne v_*}\widetilde V_w^{(v_*)}.
 \end{aligned}
$$
The first inequality uses $2Z_v\le Y_{v,v+1}+Y_{v,v-1}$.
For the second, all coefficients $c/2-\alpha_w$ are positive,
so \eqref{eq:column-row-cost} applies. Consequently
$$
 - (J^{(v_*)})'(s)
       \ge\frac1{32}\sum_w\widetilde V_w^{(v_*)}(s).
 \eqtag{eq:combined-energy}
$$
This inequality is used only between change scales.

\smallskip
\noindent\emph{The initial energy and its changes.}
For $\sigma\ge1$, comparison with the integral gives
$\sum_i g_\sigma(i-p)\le1+\|g_\sigma'\|_1<2$.
All vertex variances $\sigma_v$ are at least one on $[2,s_+]$. Consequently
$\|W_v\|_{\HS}\le2$, $\|R_v\|_{\op}\le2\sqrt2$ and
$e_v\le2\sqrt2\|R_v\|_{\HS}^2$.
For two different vertices, $\langle w_v(i)w_w(j)\rangle_s=1$:
first integrate the missing center, then the two endpoint densities.
Thus $J^{(v_*)}(2)\le C(N_0+N_1+N_2)$.

At a change scale, freeze the Gaussian weights and interpolate
$A_0$ linearly. Every intermediate entry is bounded by $\delta$.
For a fixed $(i,j)$ put
$$
 \begin{aligned}
 X^2&=\sum_t|A_0(i,t)|^2w_1(t),&
 Y^2&=\sum_t|A_0(t,j)|^2w_0(t),\\
 H^2&=\sum_t|A_2(t,i)|^2w_2(t),&
 Q^2&=\sum_t|A_1(j,t)|^2w_2(t).
 \end{aligned}
$$
Here $A_0$ denotes the intermediate array. The grid bound gives
$X,Y\le\sqrt2\delta$ and $H,Q\le\sqrt2$.
In $R_0=[W_0,W_2^*]$, the row and column norms at the changed
entry are bounded by $\sqrt{w_0(i)}(X+H)$ and
$\sqrt{w_1(j)}Y$. In $R_1=[W_1,W_0^*]$, the corresponding
bounds are $\sqrt{w_1(j)}(Y+Q)$ and $\sqrt{w_0(i)}X$.
The increment of the weighted entry is
$\sqrt{w_0(i)w_1(j)}\Delta A_0(i,j)$. The segment bound in
Lemma~\ref{lem:calculus} therefore gives
$$
 |\Delta e_0|+|\Delta e_1|
   \le C\delta(1+\delta)
       \sum_{i,j}w_0(i)w_1(j)|\Delta A_0(i,j)|.
$$
The third row matrix is unchanged. Integrate the weights and
sum the entrywise changes. Since $\delta\le1$, this gives, for
each distinguished vertex,
$$
 \sum_\tau|\Delta J^{(v_*)}(\tau)|
       \le C\delta\sum_{i,j,\tau}|\Delta A_0^\tau(i,j)|
       \le Cm\delta^2N_0.
$$
Integrating \eqref{eq:combined-energy} between changes, including
these jumps and discarding the nonnegative final energy, gives
$$
 \int_2^{s_+}\sum_w\widetilde V_w^{(v_*)}(s)\,ds
 \le32J^{(v_*)}(2)+32\sum_\tau|\Delta J^{(v_*)}(\tau)|
 \le C(N_0+N_1+N_2).
 \eqtag{eq:row-cost-budget}
$$
Here $m\delta^2=1$.

\smallskip
\noindent\emph{Recovering the lattice kernel.}
The three profiles factor as
$$
 g_3'''=g_1'*g_1'*g_1',\qquad
 g_3^{(4)}=g_1''*g_1'*g_1',\qquad
 g_3^{(5)}=g_1''*g_1''*g_1'.
$$
Denote the chosen windows by $\ell_v\in\{g_1',g_1''\}$ and put
$f_v^s=(\ell_v)_{\sqrt s}$.
Temporarily use variance $9s/8$ at every vertex and insert
$$
 D_v=\operatorname{diag}_i
              \frac{f_v^s(i-p_v)}{g_{9s/8}(i-p_v)}.
$$
Their norms are uniformly bounded: each ratio is a fixed polynomial
in $(i-p_v)/\sqrt s$ times a decaying Gaussian.
For a fixed $v$, narrow its row variance to $7s/8$ and displace its
center by $d$. Denote its Gram and score insertion by $T_v^d,U_v^d$,
and put $d\nu_s(d)=g_{s/4}(d)\,dd$. Gaussian convolution gives
$$
 T_v=\int T_v^d\,d\nu_s(d),\qquad
 R_v^*D_vR_v=\int c_v(d,s)U_v^d\,d\nu_s(d),
$$
where
$$
 c_v(d,s)=
 \begin{cases}
 -\sqrt s\,g_{s/8}(d)/g_{s/4}(d),&\ell_v=g_1',\\
 -s\,g_{s/8}'(d)/g_{s/4}(d),&\ell_v=g_1''.
 \end{cases}
$$
Indeed, $g_{7s/8}*g_{s/4}=g_{9s/8}$ gives the Gram identity.
The row density times its score is $-g_{7s/8}'$; convolving it
against $c_vg_{s/4}$ gives $\sqrt s\,g_s'$ or $s g_s''$,
since $7/8+1/8=1$. These are exactly the windows $f_v^s$.
The ratio $g_{s/8}(d)/g_{s/4}(d)$ is
$\sqrt2\exp(-2d^2/s)$, and $g_{s/8}'(d)=-(8d/s)g_{s/8}(d)$.
Thus both coefficients are bounded by $C\sqrt s$.
The Grams have a common support because their row weights are
positive. Lemma~\ref{lem:calculus} gives
$$
 b_{T_v}(R_v^*D_vR_v)
        \le Cs\int b_{T_v^d}(U_v^d)\,d\nu_s(d).
$$

Let $\Pi_d=\{p_0+p_1+p_2=d\}$, with coordinate measure
$d\pi_d=dp_0\,dp_1$ and $p_2=d-p_0-p_1$.
Any two centers may be free, since changes between these coordinates
have Jacobian of absolute value one. Moving the selected center by
$d$ sends $\Pi_0$ to $\Pi_d$ with this measure unchanged, and
$$
 \int_\R\int_{\Pi_d}H\,d\pi_d\,d\nu_s(d)
                      =\langle H\rangle_s.
$$
For the selected vertex $v$, the weights in this bracket are
exactly those defining $\widetilde V_v^{(v)}$.
Thus Lemma~\ref{lem:triangle} and \eqref{eq:row-cost-budget} yield
$$
 \int_2^{s_+}\int_{\Pi_0}
 |\tr(D_0W_0D_1W_1D_2W_2)|\,d\pi_0\,\frac{ds}{s}
 \le C\sum_v\int_2^{s_+}\widetilde V_v^{(v)}(s)\,ds
 \le C(N_0+N_1+N_2).
$$
The factor $s$ in the comparison cancels the scale denominator.
Expand the trace and substitute $x=i-p_0$, $y=j-p_1$ in its
center integral. Its $(i,j,k)$ term becomes
$$
 A_0^s(i,j)A_1(j,k)A_2(k,i)
    (f_0^s*f_1^s*f_2^s)(i+j+k)
 =A_0^s(i,j)A_1(j,k)A_2(k,i)\psi_{\sqrt s}(i+j+k).
$$
The sums are finite, and Gaussian decay on the compact scale interval
justifies the integrations. The integral of the trace before taking
absolute values is exactly $\delta$ times the form in the proposition.
Division by $\delta=m^{-1/2}$ proves the claim.
\end{proof}

\section{Kernel increments and transference}\label{sec:transference}

For a sequence $(A_N)_{N\in\N}$ of measurable real-valued functions
and integers $m,M\ge1$, define
$$
 D_m^M(A)=
 \sup_{\substack{M\le a_1<b_1\le\cdots\le a_k<b_k\\
                 a_\ell,b_\ell\in\N,\ k\le m}}
       \sum_{\ell=1}^k|A_{b_\ell}-A_{a_\ell}|,
 \qquad D_m(A)=D_m^1(A).
$$
The empty list contributes zero. These are countable suprema of
measurable functions and hence are measurable.

\begin{lemma}\label{lem:jump}
If $\int D_m(A)\,d\mu=o(m)$, then $A_N$ has a finite limit
almost everywhere as $N\to\infty$.
\end{lemma}
\begin{proof}
For $\varepsilon>0$, let
$$
 E_\varepsilon=
 \bigcap_{M=1}^{\infty}
       \bigcup_{\substack{b>a\ge M\\a,b\in\N}}
             \{|A_b-A_a|>\varepsilon\}.
$$
This set is measurable. At any point of $E_\varepsilon$, choose
each next interval beyond the preceding right endpoint.
Then $D_m(A)\ge m\varepsilon$ for every $m$, and
$$
 \varepsilon\mu(E_\varepsilon)
       \le\frac1m\int D_m(A)\,d\mu\longrightarrow0.
$$
Failure of the Cauchy property implies membership in some
$E_{1/j}$, $j\in\N$. Outside their union the sequence is Cauchy
in $\R$, so it has a finite limit.
\end{proof}

Put
$$
 \gamma=g_3,\qquad K=\gamma-\frac32\gamma'',\qquad
 H=\gamma^{(3)},\qquad
 \mathcal P=\operatorname{span}_{\R}\{K_r,H_r:r>0\}.
$$
Thus $\mathcal P$ consists of finite real linear combinations.
Every profile in $\mathcal P$ is a Schwartz function.

\begin{lemma}\label{lem:gaussian-increments}
For each $P\in\mathcal P$ there are constants $R_P\ge4$ and
$C_P<\infty$ with the following property.
With the arrays and coefficients of Proposition \ref{prop:lattice},
and at most $m$ disjoint intervals at each output entry, all with
left endpoints at least $R_P$, one has
$$
 \left|\sum_{i,j,k}A_0(i,j)A_1(j,k)A_2(k,i)
       \sum_\ell\beta_\ell(i,j)
       (P_{b_\ell(i,j)}-P_{a_\ell(i,j)})(i+j+k)\right|
 \le C_P\sqrt m\,(N_0+N_1+N_2).
$$
The constants do not depend on the endpoints, arrays or $m$.
\end{lemma}
\begin{proof}
Differentiating $t\gamma=-3\gamma'$ and the dilation formula gives
$$
 R\partial_R(\gamma^{(j)})_R
       =j(\gamma^{(j)})_R+3(\gamma^{(j+2)})_R.
$$
Apply this with $j=0,2,3$ and integrate in $dR/R$ from $a$ to $b$.
The substitution $s=R^2$ gives
$$
 K_b-K_a=-\frac94I^{\gamma^{(4)}}_{a,b},\qquad
 H_b-H_a=\frac32\left(I^{\gamma^{(3)}}_{a,b}
                             +I^{\gamma^{(5)}}_{a,b}\right).
$$
Proposition \ref{prop:lattice} therefore proves the assertion
for $K,H$ with $R_P=4$.
For $P=\sum_{j=1}^Jc_j(h_j)_{r_j}$, where $h_j\in\{K,H\}$,
use $(P_R)=\sum_jc_j(h_j)_{r_jR}$.
Choose $R_P\ge4$ so that $r_jR_P\ge4$ for every $j$.
For each $j$, multiplying all endpoints of a selected list by
$r_j$ preserves its disjointness. Apply the estimate for $h_j$,
multiply by $|c_j|$, and add. For $P=0$ the assertion is immediate.
\end{proof}

Write $b(t)=\1_{(0,1]}(t)$ on $\R$. Its samples at integer radii are
$$
 b_N(n)=\frac1N\1_{\{1,\ldots,N\}}(n),\qquad n\in\Z.
$$
The next lemma allows an $\ell^1$ error in this kernel.
Throughout its statement and proof, $S,Q$ are commuting
measure-preserving transformations, and $F,G$ are real
representatives bounded everywhere in absolute value by one.
Write $B(F,G)=(B_N(F,G))_{N\in\N}$ for the sequence of averages.

\begin{lemma}\label{lem:transfer}
Let $P\in\mathcal P$, let $\varepsilon>0$, and let $M\ge R_P$
be an integer such that
$$
 \sum_{n\in\Z}|b_N(n)-P_N(n)|\le\varepsilon
       \qquad\text{for every integer }N\ge M.
 \eqtag{eq:kernel-approximation}
$$
Then
$$
 \int D_m^M(B(F,G))\,d\mu
       \le C_P\sqrt m+4m\varepsilon,\qquad m\in\N.
 \eqtag{eq:sharp-increment-bound}
$$
No invertibility is assumed.
\end{lemma}
\begin{proof}
For this proof abbreviate $B_N=B_N(F,G)$ and $D_m^M=D_m^M(B)$.
Since $|B_N|\le1$, we have $D_m^M\le2m$.
Enumerate the lists of at most $m$ ordered disjoint integer
intervals with endpoints at least $M$.
Where $D_m^M>0$, select the first list whose sum of absolute
increments exceeds $D_m^M/2$; where $D_m^M=0$, select the empty list.
Each sum and $D_m^M$ are measurable, so this selection is measurable.
For each selected increment choose its sign, taking the sign one
at zero. Denote these signs by $\beta_\ell(x)$ and put
$$
 \begin{aligned}
 \kappa_x^b(n)&=\sum_\ell\beta_\ell(x)
                         (b_{b_\ell(x)}-b_{a_\ell(x)})(n),\\
 \kappa_x^P(n)&=\sum_\ell\beta_\ell(x)
                         (P_{b_\ell(x)}-P_{a_\ell(x)})(n).
 \end{aligned}
$$
These kernels are measurable in $x$ at every $n$.
The first is finitely supported, vanishes for $n\le0$, and satisfies
$$
 \begin{aligned}
 \sum_{n\ge1}\kappa_x^b(n)F(S^nx)G(Q^nx)&\ge\frac12D_m^M(x),\\
 \sum_n|\kappa_x^b(n)|&\le2m,\\
 \sum_n|\kappa_x^P(n)-\kappa_x^b(n)|&\le2m\varepsilon.
 \end{aligned}
$$
The last bound applies \eqref{eq:kernel-approximation} to both
endpoints of every selected interval.

For an integer $L\ge1$, put $C_L=\{0,\ldots,L\}^2$.
For $z\in X$, define
$$
 \begin{aligned}
 A_0(i,j)&=\1_{C_L}(-i,-j),\\
 A_1(j,k)&=
 \begin{cases}
 F(S^{j+k}Q^{-j}z),&(j+k,-j)\in C_{2L},\\
 0,&\text{otherwise},
 \end{cases}\\
 A_2(k,i)&=
 \begin{cases}
 G(S^{-i}Q^{k+i}z),&(-i,k+i)\in C_{2L},\\
 0,&\text{otherwise}.
 \end{cases}
 \end{aligned}
$$
Every power used in these definitions is nonnegative.
At an entry with $A_0(i,j)=1$, attach the list selected at
$S^{-i}Q^{-j}z$; use empty lists elsewhere.
The maps $(i,j)\mapsto(-i,-j)$,
$(j,k)\mapsto(j+k,-j)$ and $(k,i)\mapsto(-i,k+i)$
are bijections of $\Z^2$. The support counts are therefore at most
$|C_L|,|C_{2L}|,|C_{2L}|$.

Let $\Lambda_P(z)$ and $\Lambda_b(z)$ be the array forms with
the selected kernels $\kappa^P$ and $\kappa^b$, respectively.
Lemma \ref{lem:gaussian-increments} gives
$$
 |\Lambda_P(z)|\le C_P\sqrt m\,(|C_L|+2|C_{2L}|).
$$
For each $(i,j)$ the map $k\mapsto n=i+j+k$ is a bijection of
$\Z$. Since $|A_1|,|A_2|\le1$, the kernel error gives
$$
 |\Lambda_b(z)-\Lambda_P(z)|\le2m\varepsilon|C_L|.
$$
This comparison includes all $n\in\Z$. In particular, it absorbs
the $n\le0$ contributions of $\kappa^P$, since $\kappa^b$ is zero
there.

Now write $p=-i$, $q=-j$, $n=i+j+k$ in $\Lambda_b$.
Only $n\ge1$ occurs, and commutativity gives
$$
 F(S^{p+n}Q^qz)G(S^pQ^{q+n}z)
       =F(S^n(S^pQ^qz))G(Q^n(S^pQ^qz)).
$$
The input cutoffs are $\1_{C_{2L}}(p+n,q)$ and
$\1_{C_{2L}}(p,q+n)$. Both are one when $(p,q)\in C_L$ and
$1\le n\le L$. Their product lies in $[0,1]$.
Thus the difference between $\Lambda_b(z)$ and the sum of the
selected increments over the points $S^pQ^qz$, $(p,q)\in C_L$,
has absolute value at most
$$
 \sum_{(p,q)\in C_L}\sum_{n>L}|\kappa^b_{S^pQ^qz}(n)|.
$$
Integrate in $z$, use preservation of $\mu$, and divide by $|C_L|$.
The preceding estimates give
$$
 \frac12\int D_m^M\,d\mu
 \le C_P\sqrt m\left(1+2\frac{|C_{2L}|}{|C_L|}\right)
       +2m\varepsilon
       +\int\sum_{n>L}|\kappa_x^b(n)|\,d\mu(x).
$$
For each $x$, the last tail tends to zero; it is bounded by $2m$.
Dominated convergence applies. Since
$|C_{2L}|/|C_L|\to4$, letting $L\to\infty$ and absorbing the
absolute factors into $C_P$ proves \eqref{eq:sharp-increment-bound}.
\end{proof}

\begin{remark}\label{rem:bounded-variation}
The same array estimate gives quantitative variation bounds for
the two smooth kernels. Take real inputs with $|F|,|G|\le1$,
suppose here that $S,Q$ are invertible, and put
$$
 W_R^h(x)=\sum_{n\in\Z}h_R(n)F(S^nx)G(Q^nx),
       \qquad h=K,H,\quad R\ge4.
$$
For a Schwartz profile, comparison on unit intervals gives
$$
 \sum_n|h_R(n)|\le\|h\|_1+R^{-1}\|h'\|_1.
$$
On each compact interval of radii, Gaussian decay also gives a
common summable majorant. Hence these series have continuous
sample paths. Let $\mathcal D_m(W^h)$ denote the supremum of
sums of at most $m$ absolute increments over ordered disjoint intervals
in $[4,\infty)$. Rational endpoints give the same supremum.
The measurable selection used above supplies a signed kernel
$\kappa_x$ whose signed sum is at least half this supremum, with
$\sum_n|\kappa_x(n)|\le Cm$.

Use the same array formulas with the symmetric boxes
$C_L=\{-L,\ldots,L\}^2$ and $C_{2L}$.
All powers are defined by invertibility. For $|n|\le L$ both
input cutoffs equal one. There is no approximation error, and the
preceding argument gives
$$
 \frac12\int\mathcal D_m(W^h)\,d\mu
 \le C\sqrt m\left(1+2\frac{|C_{2L}|}{|C_L|}\right)
       +\int\sum_{|n|>L}|\kappa_x(n)|\,d\mu(x).
$$
The tail is dominated by $Cm$ and tends to zero. Thus
$\int\mathcal D_m(W^h)\,d\mu\le C\sqrt m$.
For $\rho>2$, define $V^\rho(W^h)$ as the supremum of the
$\ell^\rho$ norms of consecutive increments over finite increasing
partitions of $[4,\infty)$. This function is measurable by continuity.
If the absolute increments of one partition in decreasing order
are $d_1\ge d_2\ge\cdots$, then $kd_k\le\mathcal D_k(W^h)$.
Grouping the ranks $2^j\le k<2^{j+1}$ yields
$$
 V^\rho(W^h)\le
       \sum_{j\ge0}2^{-j(1-1/\rho)}\mathcal D_{2^j}(W^h).
$$
After integration, the right side is bounded by
$C\sum_{j\ge0}2^{-j(1/2-1/\rho)}<\infty$.
This proves $\|V^\rho(W^K)\|_1+\|V^\rho(W^H)\|_1\le C_\rho$
for the normalized real inputs. For complex inputs, expand into
the four real pairs and use the triangle inequality for $V^\rho$.
Rescaling then gives a bound of
$C_\rho\|F\|_\infty\|G\|_\infty$ for arbitrary bounded inputs.
\end{remark}

\section{Approximation of the forward averaging kernel}\label{sec:recovery}

\begin{lemma}\label{lem:recovery}\label{lem:kernel-recovery}
For every $\varepsilon>0$, there exist $P\in\mathcal P$ and an
integer $M\ge R_P$ such that
$$
 \sum_{n\in\Z}|b_N(n)-P_N(n)|\le\varepsilon
       \qquad(N\in\N,\ N\ge M).
$$
\end{lemma}
\begin{proof}
We first show that the positive dilates of $K$, and separately
those of $H$, span dense subspaces of $L^1(0,\infty)$.
For $q=0,1$, put
$$
 h_q(t)=t^q\left(\frac32-\frac{t^2}{6}\right)e^{-t^2/6}.
$$
On the positive half-line,
$K=\gamma(0)h_0$ and $H=(2\gamma(0)/9)h_1$.
It therefore suffices to prove density for each $h_q$.

Let $f\in L^\infty(0,\infty)$ annihilate all dilates of $h_q$,
and set
$$
 F(a)=\int_0^\infty f(t)t^qe^{-at^2}\,dt,\qquad a>0.
$$
On compact intervals of $a>0$, the integrands and their first
$a$ derivatives have an integrable Gaussian majorant.
Thus differentiation under the integral is valid.
Set $a=1/(6R^2)$ in the annihilation identity and remove its
nonzero dilation factor. We obtain
$$
 aF'(a)+\frac32F(a)=0.
$$
Consequently $(a^{3/2}F(a))'=0$, so $F(a)=C_0a^{-3/2}$.
On the other hand, the substitution $u=\sqrt a\,t$ gives
$$
 |F(a)|\le\|f\|_\infty a^{-(q+1)/2}
                  \int_0^\infty u^qe^{-u^2}\,du.
$$
Multiplying by $a^{3/2}$ and letting $a\downarrow0$ gives
$C_0=0$, since $q=0,1$. Hence $F(a)=0$ for every $a>0$.

Push the finite complex measure $f(t)t^qe^{-t^2}\,dt$ forward by
$t\mapsto e^{-t^2}$, and regard it as a measure $\nu$ on $[0,1]$.
For every integer $j\ge0$,
$$
 \int_{[0,1]}z^j\,d\nu(z)=F(j+1)=0.
$$
Polynomials are uniformly dense in $C[0,1]$, so $\nu$ integrates
every continuous function to zero and is the zero measure.
The map $t\mapsto e^{-t^2}$ is a bijection onto $(0,1)$ with
measurable inverse. Its original measure is therefore zero.
Since $t^qe^{-t^2}>0$ for $t>0$, this gives $f=0$ almost everywhere.
If the closed span of the dilates were proper, Hahn--Banach and
$L^1$--$L^\infty$ duality would supply a nonzero bounded annihilator.
The preceding argument excludes one, proving the asserted density.

On $(0,\infty)$ choose finite real linear combinations $E$ of
dilates of $K$ and $O$ of dilates of $H$ such that
$$
 \|E-b/2\|_{L^1(0,\infty)}<\varepsilon/8,\qquad
 \|O-b/2\|_{L^1(0,\infty)}<\varepsilon/8.
$$
Real coefficients suffice by taking real parts of any approximants.
The same formulas on $\R$ make $E$ even and $O$ odd.
For $P=E+O\in\mathcal P$,
$$
 \|P-b\|_{L^1(\R)}
 \le2\|E-b/2\|_{L^1(0,\infty)}
      +2\|O-b/2\|_{L^1(0,\infty)}
 <\varepsilon/2.
$$

It remains to pass to lattice samples.
For an integer $N\ge1$, let $P^{[N]}(t)=P(n/N)$ on
$((n-1)/N,n/N]$, for every $n\in\Z$.
On the same interval, $b(t)=b(n/N)$ almost everywhere.
The fundamental theorem of calculus gives
$$
 \|P-P^{[N]}\|_1\le\frac1N\|P'\|_1.
$$
Indeed, on each interval the integral of the endpoint error is at
most its length $1/N$ times the integral of $|P'|$ there.
Consequently,
$$
 \begin{aligned}
 \sum_{n\in\Z}|b_N(n)-P_N(n)|
   &=\int_\R|b(t)-P^{[N]}(t)|\,dt\\
   &\le\|b-P\|_1+\frac1N\|P'\|_1.
 \end{aligned}
$$
Choose an integer $M\ge\max\{R_P,\,2\|P'\|_1/\varepsilon,\,1\}$.
Then the last expression is at most $\varepsilon$ for every
integer $N\ge M$, as required.
\end{proof}

\section{Integrable functions and applications}\label{sec:extensions}

\subsection{Proof of Theorem~\ref{thm:main}}

First take real bounded inputs, and normalize their essential
suprema to be at most one. Choose representatives bounded everywhere
by one. Fix $\varepsilon>0$, and choose $P,M$ from
Lemma \ref{lem:kernel-recovery}.
Lemma \ref{lem:transfer} gives
$$
 \int D_m^M(B(F,G))\,d\mu\le C_P\sqrt m+4m\varepsilon.
$$
For any selected list of integer intervals, split an interval at
$M$ if it crosses that point. There is at most one such interval.
The portions below $M$ have integer endpoints and are disjoint,
so there are at most $M-1$ of them. Since $|B_N(F,G)|\le1$, their
total contribution is at most $2M$.
The portions above $M$ form at most $m$ disjoint intervals.
The triangle inequality therefore gives
$$
 D_m(B(F,G))\le D_m^M(B(F,G))+2M.
$$
Keep $\varepsilon$, $P$ and $M$ fixed and let $m\to\infty$.
We obtain
$$
 \limsup_{m\to\infty}\frac1m\int D_m(B(F,G))\,d\mu
       \le4\varepsilon.
$$
Since $\varepsilon>0$ was arbitrary, this integral is $o(m)$.
Lemma \ref{lem:jump} proves convergence of the forward averages.
All transformations used in this argument have nonnegative powers,
so the bounded conclusion already holds without invertibility.
Writing each complex input as its real part plus $i$ times its
imaginary part expresses $B_N(F,G)$ as a linear combination of
four averages for real pairs. Intersecting their conull convergence
sets proves the bounded complex case.

It remains to extend the input range.
If $p,q<\infty$, put, for $M\in\N$,
$F_M=F\1_{\{|F|\le M\}}$ and $G_M=G\1_{\{|G|\le M\}}$.
Write $\mathcal I_S,\mathcal I_Q$ for the invariant sigma-algebras.
Discrete H\"older and Birkhoff give, on a common conull set,
$$
 \begin{aligned}
 \limsup_N|B_N(F,G)-B_N(F_M,G_M)|
 \le{}&
 \E(|F-F_M|^p\mid\mathcal I_S)^{1/p}
             \E(|G|^q\mid\mathcal I_Q)^{1/q}\\
 &+\E(|F|^p\mid\mathcal I_S)^{1/p}
             \E(|G-G_M|^q\mid\mathcal I_Q)^{1/q}.
 \end{aligned}
$$
H\"older applies because $1/p+1/q\le1$, using the constant function
$1$ for a possible remaining exponent.
By \eqref{eq:conditional-tails}, the conditional expectations of
the tails decrease to zero almost everywhere. The right side therefore tends to zero.
For fixed $M$, convergence of the truncated averages and the triangle
inequality give
$$
 \limsup_{N,N'\to\infty}|B_N(F,G)-B_{N'}(F,G)|
 \le2\limsup_{N\to\infty}|B_N(F,G)-B_N(F_M,G_M)|.
$$
Letting $M\to\infty$ proves the Cauchy property almost everywhere.
If $p<\infty$ and $q=\infty$, only truncate $F$. For every $N$,
$$
 |B_N(F,G)-B_N(F_M,G)|
 \le\|G\|_\infty
       \left(\frac1N\sum_{n=1}^N|F-F_M|^p(S^nx)\right)^{1/p}.
$$
Birkhoff and conditional monotone convergence give
$$
 \limsup_N|B_N(F,G)-B_N(F_M,G)|
 \le\|G\|_\infty\E(|F-F_M|^p\mid\mathcal I_S)^{1/p}
       \longrightarrow0\quad(M\to\infty).
$$
If $p=\infty$ and $q<\infty$, truncating only $G$ gives
$$
 \limsup_N|B_N(F,G)-B_N(F,G_M)|
 \le\|F\|_\infty\E(|G-G_M|^q\mid\mathcal I_Q)^{1/q}
       \longrightarrow0\quad(M\to\infty),
$$
by the same finite-sum inequality with the roles of the inputs
interchanged. The truncated pairs are bounded in both cases, so the
preceding Cauchy estimate applies. These estimates include
$(p,q)=(1,\infty)$ and $(\infty,1)$. The case
$p=q=\infty$ was already proved, completing the argument.

\subsection{Sharpness and joint-ergodicity applications}\label{sec:applications}

The range of exponents in Theorem~\ref{thm:main} is optimal when the
conclusion is required for every commuting pair.

\begin{proposition}\label{prop:range}
Let $1\le p,q<\infty$ and $1/p+1/q>1$. There are an invertible
mixing probability-preserving transformation $T$ and nonnegative
functions $F\in L^p$, $G\in L^q$ such that
$$
 \frac1N\sum_{n=1}^N F(T^nx)G(T^nx)\longrightarrow+\infty
 \quad\text{for almost every }x.
$$
\end{proposition}

\begin{proof}
Choose $a,b>0$ with $a<1/p$, $b<1/q$, and $a+b\ge1$.
Let $T$ be the two-sided Bernoulli shift on $[0,1]^{\Z}$ with
product Lebesgue measure, and let $U(x)=x_0$. Put
$F=U^{-a}$ and $G=U^{-b}$, assigning arbitrary finite values on
the null set $\{U=0\}$. Then $F\in L^p$, $G\in L^q$, and
$H=FG$ is nonnegative with infinite integral.

For each integer
$m\ge1$, Theorem~\ref{thm:ergodic-prelim} gives
$$
 \liminf_{N\to\infty}\frac1N\sum_{n=1}^N H(T^nx)
 \ge \int\min(H,m)\,\dd\mu
$$
on a common conull set. The truncations $\min(H,m)$ are integrable and the Bernoulli shift
is ergodic, so Theorem~\ref{thm:ergodic-prelim} applies to every $m$. Monotone
convergence gives $\int\min(H,m)\,\dd\mu\to\infty$. Taking
the supremum over $m$ proves the assertion with $S=Q=T$.
\end{proof}

Pointwise convergence need not imply convergence in $L^\infty$,
even for bounded inputs. On the same Bernoulli shift take
$F=\one_{\{x_0>1/2\}}$ and $G=1$. The limit is $1/2$ almost
everywhere. For every $N$, the cylinder on which the first $N$
sampled coordinates exceed $1/2$ has positive measure, and the
average equals $1$ there. Its $L^\infty$ distance from the limit
is therefore $1/2$ for every $N$.

For general commuting maps the limit need not be constant. The
following sufficient conditions identify it as a product of
integrals. We include the norm argument to specify exactly which ergodicity
assumptions are used. The assumption on $S^{-1}Q$ separates the two
factors in each fixed correlation; the assumption on $S\times Q$
then identifies the limit of those correlations.

\begin{corollary}\label{cor:joint}
Let $S,Q$ be commuting invertible measure-preserving transformations
of a probability space $(X,\cB,\mu)$. Suppose that $S^{-1}Q$ is ergodic and that $S\times Q$ is ergodic
on the product probability space. If $F\in L^p(\mu)$ and
$G\in L^q(\mu)$, where $1\le p,q\le\infty$ and
$1/p+1/q\le1$, then
$$
 B_N(F,G)(x)\longrightarrow
       \left(\int F\,\dd\mu\right)
       \left(\int G\,\dd\mu\right)
 \quad\text{for almost every }x.
$$
The convergence also holds in $L^r(\mu)$ for every $1\le r<\infty$
with $1/r\ge1/p+1/q$.
\end{corollary}

\begin{proof}
The proof is divided into two steps.

\proofstep{1}{Reduce to correlations for bounded inputs.}
First assume that $F,G$ are bounded, and put
$$
 \begin{aligned}
 R&=S^{-1}Q,\qquad u_n=(F\circ S^n)(G\circ Q^n),\\
 c&=\left(\int F\,\dd\mu\right)\left(\int G\,\dd\mu\right).
 \end{aligned}
$$
Invariance gives
$$
 \int u_n\,\dd\mu=\int F\,(G\circ R^n)\,\dd\mu.
$$
Average this identity over $n$ and apply
Theorem~\ref{thm:ergodic-prelim} to $G$ under the ergodic transformation $R$.
It follows that
$\int B_N(F,G)\,\dd\mu\to c$.

Fix $H\ge1$. The boundary terms give
$$
 \begin{aligned}
 &\left\|B_N(F,G)-\frac1{NH}
       \sum_{n=1}^N\sum_{h=1}^H u_{n+h}\right\|_{L^2}\\
 &\qquad\le\frac{H+1}{N}\|F\|_{L^\infty}\|G\|_{L^\infty}.
 \end{aligned}
$$
Indeed, each shift by $h$ changes at most $2h$ boundary terms,
and we average these bounds over $1\le h\le H$.
Convexity of the squared Hilbert norm therefore gives
$$
 \limsup_{N\to\infty}\norm{B_N(F,G)}_{L^2}^2
 \le \lim_{N\to\infty}\frac1N\sum_{n=1}^N
               \norm{\frac1H\sum_{h=1}^H u_{n+h}}_{L^2}^2.
 \eqtag{eq:joint-averaging-bound}
$$
For $1\le h,k\le H$, a change of variables by $S^{n+k}$ gives
$$
\begin{aligned}
 \int u_{n+h}\overline{u_{n+k}}\,\dd\mu
 =\int &(F\circ S^{h-k})\overline F\\
       &\cdot\bigl((G\circ Q^{h-k})\overline G\bigr)
                         \circ R^{n+k}\,\dd\mu.
\end{aligned}
$$
By Theorem~\ref{thm:ergodic-prelim}, the Ces\`aro mean of these correlations in $n$ tends to
$$
 \left(\int(F\circ S^{h-k})\overline F\,\dd\mu\right)
 \left(\int(G\circ Q^{h-k})\overline G\,\dd\mu\right).
$$
\proofstep{2}{Identify the limit and extend the inputs.}
Expanding the squared norm in \eqref{eq:joint-averaging-bound}
and using these limits for each pair $h,k$ gives
$$
 \limsup_{N\to\infty}\norm{B_N(F,G)}_{L^2}^2
 \le
 \norm{\frac1H\sum_{h=1}^H
                 (F\otimes G)\circ(S\times Q)^h}_{L^2(\mu\times\mu)}^2.
 \eqtag{eq:joint-product-bound}
$$
The length $H$ is held fixed while $N$ tends to infinity. Thus
the first mean ergodic argument concerns only finitely many fixed
correlations. It requires no uniform estimate for a growing family
of functions. The second limit, in which $H$ tends to infinity,
is a single ordinary mean ergodic average on the product space.
Apply Theorem~\ref{thm:ergodic-prelim} to $S\times Q$ in \eqref{eq:joint-product-bound}. Its ergodicity makes the right-hand side tend to $|c|^2$
as $H\to\infty$. Together with convergence of the integrals, this
implies $\norm{B_N(F,G)-c}_{L^2}\to0$. The pointwise convergence
from Theorem~\ref{thm:main} then identifies the almost-everywhere limit
as $c$.

For general $F,G$, use the bounded truncations $F_m,G_m$ from the
proof of Theorem~\ref{thm:main}; an input with infinite exponent is left
unchanged. Set
$$
 c_m=\left(\int F_m\,\dd\mu\right)
     \left(\int G_m\,\dd\mu\right),
 \qquad
 c=\left(\int F\,\dd\mu\right)
   \left(\int G\,\dd\mu\right).
$$
The truncations converge in their respective norms and hence in
$L^1(\mu)$, so $c_m\to c$. Fix $1\le r<\infty$ with
$1/r\ge1/p+1/q$. H\"older's inequality and invariance give
$$
\begin{aligned}
 \sup_N\norm{B_N(F,G)-B_N(F_m,G_m)}_{L^r}
 &\le \norm{F-F_m}_{L^p}\norm{G}_{L^q}\\
 &\quad+\norm{F_m}_{L^p}\norm{G-G_m}_{L^q}
 \longrightarrow0.
\end{aligned}
$$
For each fixed $m$, the bounded case already proved gives
$B_N(F_m,G_m)\to c_m$ almost everywhere. Since
$|B_N(F_m,G_m)|\le\norm{F_m}_{L^\infty}\norm{G_m}_{L^\infty}$,
dominated convergence gives the same limit in $L^r(\mu)$.
Consequently,
$$
\begin{aligned}
 \limsup_{N\to\infty}\norm{B_N(F,G)-c}_{L^r}
 &\le \sup_N\norm{B_N(F,G)-B_N(F_m,G_m)}_{L^r}\\
 &\quad+|c_m-c|
 \longrightarrow0\qquad(m\to\infty).
\end{aligned}
$$
This proves the asserted norm convergence. By Theorem~\ref{thm:main}, the
original averages have an almost-everywhere limit, which must equal
$c$ because they converge to $c$ in $L^r(\mu)$.
\end{proof}

\begin{corollary}\label{cor:independent}
Let $m\ge1$, and let $(X_j,\mathcal B_j,\mu_j,R_j)$,
$1\le j\le m$, be invertible weakly mixing probability-preserving
systems. For each $j$, let $a_j,b_j$ be distinct nonzero integers
and let $F_j,G_j\in L^\infty(\mu_j)$. Then, for
$\mu_1\times\cdots\times\mu_m$-almost every $(x_1,\ldots,x_m)$,
$$
 \frac1N\sum_{n=1}^N\prod_{j=1}^m
       F_j(R_j^{a_jn}x_j)G_j(R_j^{b_jn}x_j)
 \longrightarrow
 \prod_{j=1}^m
       \left(\int F_j\,\dd\mu_j\right)
       \left(\int G_j\,\dd\mu_j\right).
$$
\end{corollary}

\begin{proof}
On the product space put
$$
 \begin{aligned}
 S&=\prod_{j=1}^mR_j^{a_j},&
 F&=\bigotimes_{j=1}^mF_j,\\
 Q&=\prod_{j=1}^mR_j^{b_j},&
 G&=\bigotimes_{j=1}^mG_j.
 \end{aligned}
$$
By \cite[Chapter~1]{Walters}, nonzero powers and finite products of
invertible weakly mixing transformations are weakly mixing. Each $R_j^{a_j}$ and $R_j^{b_j}$
is therefore weakly mixing, so $S\times Q$ is ergodic. Since
$a_j\ne b_j$ for every $j$, the map
$S^{-1}Q=\prod_{j=1}^mR_j^{b_j-a_j}$ is also weakly mixing.
Apply Corollary~\ref{cor:joint} to $S,Q,F,G$. Factoring the two integrals over the product measure proves the assertion.
\end{proof}
% \section*{Note on related work}
% After the first version of this preprint appeared on arXiv, we learned
% of the Lean formalization by Floris van Doorn, Polona Durcik,
% Joris Roos, Lenka Slav\'ikov\'a, and Christoph Thiele
% \cite{palomar-2026-10-18-000026-v1}.
% Their pointwise convergence theorem is exactly the same as
% Theorem~\ref{thm:main}.
% Roos informed us that he had this formalization on October~8, 2026,
% before we began our work.
% Both arguments use matrix energies and Gaussian estimates.
% The blueprint in that project passes through commuting real flows
% and a suspension construction. Here we transfer a finite-lattice
% estimate directly to the commuting transformations and recover the
% forward averages by $L^1(\R)$ approximation of their kernels.

\section*{Note on related work}
After the first version of this preprint appeared on arXiv, we learned
that the same pointwise convergence result as in
Theorem~\ref{thm:main} had already been proved in the Lean
formalization by Floris van Doorn, Polona Durcik, Joris Roos,
Lenka Slav\'ikov\'a, and Christoph Thiele
\cite{palomar-2026-10-18-000026-v1}.
Roos subsequently informed us that he had this formalization on
October~8, 2026, before we began our work.
We thank him for this clarification and for pointing us to
\cite{NormVariationCommuting}.

Both arguments use matrix energies and Gaussian estimates.
The blueprint in that project passes through commuting real flows
and a suspension construction. Here we transfer a finite-lattice
estimate directly to the commuting transformations and recover the
forward averages by $L^1(\R)$ approximation of their kernels.

\smallskip
\noindent\emph{Remark.}
In their August 2026 preprint \cite{NormVariationCommuting}, the same
authors prove norm-variation estimates for multiple ergodic averages
associated with $n\ge2$ commuting measure-preserving transformations.
These estimates give a quantitative strengthening of Tao's
norm-convergence theorem.

\section*{Declaration of AI use}
The authors used ChatGPT (GPT-6) during the development of this work.
The authors identified a connection between the triangular structure
arising from commuting transformations and the matrix-energy and
variation methods developed in \cite{OAIAnnular}.
Subsequent discussions with ChatGPT helped explore how these methods
could be adapted to the discrete estimates needed for pointwise
ergodic averages.
These interactions also contributed to the choice of the even
Gaussian kernel $K$, the development of the proof strategy, and
the derivation of technical details in
Lemma~\ref{lem:kernel-recovery}.
ChatGPT was also used to polish the text.
The authors take full responsibility for all mathematical arguments,
references, and conclusions presented in this paper.

\end{document}